\documentclass[11pt]{amsart}

\usepackage{amssymb}
\usepackage{amsthm}
\usepackage{amsmath}
\usepackage{graphicx}
\usepackage{comment}
\usepackage[usenames,dvipsnames]{xcolor}
\usepackage[margin = 1in]{geometry}
\usepackage{enumerate}
\usepackage[toc,page]{appendix}
\usepackage{color}
\usepackage{hyperref}

\definecolor{mygreen}{rgb}{0.1,0.75,0.2}

 \newtheorem{thm}{Theorem}[section]
 
 \newtheorem{lem}[thm]{Lemma}
 \newtheorem{prop}[thm]{Proposition}

 \theoremstyle{definition}
 \newtheorem{defn}{Definition}

 \theoremstyle{remark}

 \numberwithin{equation}{section}

\newcommand{\la}{\langle}
\newcommand{\ra}{\rangle}
\newcommand{\pd}{\partial}
\newcommand{\eps}{\varepsilon}
\newcommand{\ud}{\,\mathrm{d}}
\newcommand{\I}{\mathbb{T}}
\newcommand{\8}{\infty}
\newcommand{\lnx}{(\ln h_x)_x}
\newcommand{\lnxn}{(\ln h_{nx})_x}
\newcommand{\rhu}{{\overset{*}{\rightharpoonup}}}

\begin{document}

\title[Global BV solution]{Global Strong Solution With BV Derivatives to Singular Solid-on-Solid model With Exponential Nonlinearity}

\author{Yuan Gao}
\address{Department of Mathematics, Duke University,
  Durham NC 27708, USA}
\email{yuangao@math.duke.edu}

\date{\today}

\begin{abstract}
In this work, we consider the one dimensional very singular fourth-order  equation for solid-on-solid model in attachment-detachment-limit regime with exponential nonlinearity
$$h_t = \nabla \cdot (\frac{1}{|\nabla h|} \nabla e^{\frac{\delta E}{\delta h}}) =\nabla \cdot (\frac{1}{|\nabla h|}\nabla e^{- \nabla \cdot (\frac{\nabla h}{|\nabla h|})})$$
where total energy $E=\int |\nabla h|$ is the total variation of $h$.
Using a logarithmic correction $E=\int |\nabla h|\ln|\nabla h| \ud x$ and gradient flow structure with a suitable defined functional,
we prove the evolution variational inequality solution preserves a positive gradient $h_x$ which has upper and lower bounds but in BV space. We also obtain the global strong solution to the solid-on-solid model
which allows an asymmetric singularity $h_{xx}^+$ happens.
\end{abstract}

\maketitle

\section{Introduction}
\subsection{Background}
Epitaxial growth on crystal surface is an important nanoscale phenomena which has attracted  lots of attention due to its application in  industry and in manufacture of some typical experimental materials. We refer to \cite{SSR2, PimpinelliVillain:98} for more physical description.

In this paper, we focus on dynamic process for solid on solid (SOS) model on crystal surface, where adatoms detach from above, diffuse on the substrate and then are absorbed at another position. There are some researches on the SOS model from microscopic viewpoint and derivation of continuum limit from mesoscopic level; see \cite{Yip2001, our, Kohnbook, Margetis2006, Tang1997}.
The kinetic process can also be described using macroscopic variable, height profile $h(x,t)$ of a solid film.  Here we directly write down the evolution equation for surface height $h(x,t)$ using conservation law of mass
$$h_t+\nabla\cdot J=0,$$
where $$J=-M(\nabla h)\nabla\rho_s$$ is the adatom flux by Fick's law \cite{Margetis2006}, the mobility function $M(\nabla h)$ is a functional of $\nabla h$  and $ \rho_s$ is the local equilibrium
density of adatoms.
By the Gibbs-Thomson relation \cite{cooper1996, widom1982, Margetis2006}, which is connected to the theory of molecular capillarity, the corresponding local equilibrium
density of adatoms is given by
$$\rho_s=\rho^0 e^{\frac{\mu}{kT}},$$
where $\rho^0$
is a constant
reference density, $T$ is the temperature and $k$ is the Bolzmann constant.

Now we consider the expression of the chemical potential $\mu$, the rate of change in the surface energy per atom. For a physical constant $L$, we impose periodic boundary condition for simplicity, i.e.
\begin{equation}\label{Periodic}
  h(x+L)=h(x)+1 \quad \text{ for a.e. }x\in\mathbb{R}.
\end{equation}
 Denote the domain for one period as $\I:=[0,L)$.
The general total energy for epitaxial growth is
\begin{equation}
  E= \frac{1}{p}\int_\I |\nabla h|^p \ud x
\end{equation}
for some $p\geq 1$  and the corresponding chemical potential is
\begin{equation}
  \mu=\frac{\delta E}{\delta h} = - \nabla \cdot (|\nabla h|^{p-2} \nabla h).
\end{equation}
Hence the general evolution equation becomes
\begin{equation}\label{exp001}
h_t = \nabla \cdot \Big(M(\nabla h) \nabla e^{\frac{\mu}{KT}}\Big)=\nabla \cdot \Big(M(\nabla h) \nabla e^{- \nabla \cdot (|\nabla h|^{p-2} \nabla h)}\Big),
\end{equation}
where the mobility $M(\nabla h)$ is a constant in diffusion-limit (DL) regime, while the mobility $M(\nabla h)=\frac{1}{|\nabla h|}$ in the attachment-detachment-limit (ADL) regime; see \cite{our, our2, Kohnbook, D24, Margetis2006, Zang1990} and the references in there.
\\
{\bf Difficulties and references in the continuum framework.} In the previous researches, the exponential form of chemical potential $e^{\mu/kT}$ is regarded as linear in chemical potential $e^{\mu/kT}\approx1+\mu/kT$ under the hypothesis $|\mu|\ll kT.$ When $p>1$, we refer to \cite{She2011, Leoni2015, our2, ourxu, LX} for analytical results including existence, uniqueness and long time behaviors in DL regime and ADL regime. General speaking, the ADL model is harder than DL model due to the singular mobility $\frac{1}{|\nabla h|}$ so the global monotone solution is understood in almost everywhere sense in  \cite{our2, ourxu}.
For the case $p=1$, the total energy and chemical potential become  the total variation of $h$ (see physical derivation from mesoscopic viewpoint by bond counting in \cite{Liu-Bob}), i.e.
\begin{equation}
  E=\int_\I |\nabla h| , \quad  \mu=\frac{\delta E}{\delta h} = - \nabla \cdot (\frac{\nabla h}{|\nabla h|}).
\end{equation}
After linearization, this kind of fourth-order singular equation in DL regime is regarded as $H^{-1}$ gradient flow for the BV seminorm $\int |\nabla h|$. The discontinuous solution is studied in \cite{Giga-Giga2010} and the flattening effect in finite time is proved in \cite{giga2010}; see also \cite{Giga-Kuroda, Giga-Monika} for further development in $H^{-s}$ space and other boundary conditions. However, the method therein works only for DL regime whose mobility is a constant and the evolution in ADL regime is still an open question for $p=1$. More recently, the original exponential equation \eqref{exp001} in DL regime is studied in \cite{G2, Liu-Bob, M-n} for $p=2$ and in \cite{Xu-n} for $p\in(1,2]$, where the existence of strong solution with latent singularity and global solution starting from small data are established. For $p=1$ in DL regime, \cite{LLDM} constructs some explicit solution to demonstrate the asymmetry of height profile due to exponential effect.
No matter with or without linearization, those method in DL regime for $p=1$ more or less relies on the total variation flow structure of the PDE so it fails to work for ADL regime.  To our best knowledge, there is no result for the evolution equation in ADL regime with $p=1$
\begin{equation}\label{surface}
  h_t = \nabla \cdot (M(h) \nabla e^{\frac{\mu}{KT}}) =\nabla \cdot (\frac{1}{|\nabla h|}\nabla e^{- \nabla \cdot (\frac{\nabla h}{|\nabla h|})}),
\end{equation}
which is a very singular fourth order equation with exponential nonlinearity.
\\
{\bf Logarithmic correction and explanation from mesoscopic view.}
From the mesoscopic view  we can regard the surface evolution equation as  continuum limit of discrete Burton-Cabrera-Frank (BCF) model \cite{BCF, our, our2}, which tracks the dynamics of positions of each step $x_i$ with height $h_i= h_0+\frac{i}{N}$. In ADL regime, it can be expressed by
\begin{equation}
\frac{\ud x_i}{\ud t}=N \bigl[(f_{i+1}-f_{i})-(f_{i}-f_{i-1})\bigr],\quad i=1,\cdots,N,
\end{equation}
 with
only repulsive interaction between the nearest step
$$f_i:=-\biggl(\frac{1}{x_{i+1}-x_i}-\frac{1}{x_{i}-x_{i-1}}\biggr)=\frac{\partial E_i}{\partial x_i},$$
which is actually the dominated elastic interaction \cite{Xiang2002} in BCF step model depending on the distance between steps.
Then the corresponding discrete energy is
$$E_i=\frac{1}{2}\sum_{i=1}^N\sum_{|j-i|=1} \ln \lvert x_i-x_j\rvert, $$
where $N$ goes to $+\8$ in the continuum limit.
The corresponding continuum interaction function $f$ in the limit PDE is
$f=-\nabla(\ln|\nabla h|)$ ( see detailed consistent check in \cite{our}), which inspires us that we shall use a logarithmic factor to adjust the total energy for the case of $p=1.$
Therefore, we take the  total energy with logarithmic correction as
\begin{equation}
   E(h):= \int |\nabla h|\ln|\nabla h| \ud x, \quad  \mu:=\frac{\delta E(h)}{\delta h} = - \nabla \cdot \Big( \frac{\nabla h}{|\nabla h|} (\ln |\nabla h|+1) \Big).
\end{equation}
This kind of logarithmic correction is also used for the linearized surface evolution equation in \cite{Gao-Ji} since the logarithmic correction is negligible for small surface gradients.
The surface height equation turns out to be
\begin{equation}\label{surface}
  h_t = \nabla \cdot (M(h) \nabla e^{\frac{\mu}{KT}}) =\nabla \cdot\Big(\frac{1}{|\nabla h|}\nabla e^{- \nabla \cdot \big(\frac{\nabla h}{|\nabla h|}(\ln |\nabla h|+1)\big)}\Big).
\end{equation}
\\
{\bf{Results and methods.}}
In this paper, we start with the simplest situation: one dimensional case with monotone initial data, i.e. $\partial_x h_0>0.$
If we can prove $h_x>0$ for all the time, then we obtain a mathematical validation  for
surface height equation \eqref{surface}, i.e.
\begin{equation}\label{maineq}
    h_t = \nabla \cdot (M(h) \nabla e^{\mu}) =\left(\frac{1}{h_x}( e^{- (\ln h_x)_x})_x\right)_x
\end{equation}
with $\mu=\frac{\delta E(h)}{\delta h}=-(\ln h_x)_x.$ Specifically, we investigate the existence and uniqueness of the evolution variational inequality (EVI) solution and monotone strong solution to \eqref{surface} with a monotone initial data; see Theorem \ref{cor10} and Theorem \ref{mainth1} separately. We first observe the $L^2$ gradient flow structure by defining a proper, lower semi-continuous convex functional $\phi$. However due to the asymmetric effect brought by exponential nonlinearity, we shall allow a latent singularity for $h_{xx}$ and define the convex functional only on the absolutely continuous part of $h_{xx}$; see rigorous definition in \eqref{phi}. Then thanks to the detailed properties for the convex functional $\phi$ and the bound for $h_x$ provided by one dimensional BV space, we can apply the gradient flow method in metric space \cite{AGS} to obtain the EVI solution whose gradient is in BV space and has upper/lower bound. To further explore the strong solution with latent singularity to \eqref{maineq} in the sense that the equation holds almost everywhere, we carefully  characterize the sub-differential of $\phi$ by first carry on the calculations in some dense set then prove the sub-differential $\partial\phi$ is single-valued.
We call it latent singularity because the singularity in solution does not effect the evolution of the solution but it is not removable. In the end, the singular PDE is understood in the sense that the equation holds almost everywhere after removing  the singular part of $(\ln h_x)_x$. That is to say, the PDE is understood as a limit of a regularized problem.

\subsection{Gradient flow in $L^2(\I)$}\label{sec1.2}
Let us define a new functional with some formal observations and recast \eqref{maineq} into a $L^2(\I)$ gradient flow. Let $\phi$ be
\begin{equation}
  \phi(h):= \int_\I e^{-(\ln h_x)_x} \ud x.
\end{equation}
The variation of $\phi$ is
$$\frac{\delta \phi}{\delta h}=-\Big(\frac{1}{h_x} \big( e^{-(\ln h_x)_x} \big)_x \Big)_x$$
and then formally we have
\begin{equation}\label{gfs}
  h_t=-\frac{\delta\phi}{\delta h}.
\end{equation}

To study the monotone strong solution to \eqref{maineq}, we plan to apply the gradient flow theory in metric space, $L^2(\I)$.
It requires we  clarify the working space associated with proper topology. We will define $\phi(h)$ rigorously later in \eqref{phi}. Let us first see some inspiring observations.
\\
{\bf Observation 1.} Thanks to the periodic assumption, we have
\begin{equation}
  \frac{\ud}{\ud t} \int_\I h \ud x = 0,
\end{equation}
which implies $\int_\I h \ud x =\int_\I h_0 \ud x.$ Moreover from
$$\int_\I h_{xx} \ud x =0$$
we know
\begin{equation}
  \int_\I (h_{xx})^+ \ud x = \int_\I (h_{xx})^- \ud x.
\end{equation}
Here $(h_{xx})^- $ is the negative part of $h_{xx}$ and $(h_{xx})^+ $ is the positive part of $h_{xx}$. In fact, the notation of integration is just formal for now and we will see $(h_{xx})^+$ could be Radon measure later.
\\
{\bf Observation 2.}
From the gradient flow structure \eqref{gfs},
\begin{equation}
  \frac{\ud \phi}{\ud t} =\int_\I \frac{\delta\phi}{\delta h} h_t \ud x = -\int_\I \big|\frac{\delta\phi}{\delta h}\big|^2 \ud x =-\int_\I h_t^2 \ud x \leq 0,
\end{equation}
which gives the observation
$$\phi(u(t))\leq \phi(u(0)) \quad \text{ for any }t\geq 0.$$
Therefore we obtain uniform estimate
\begin{align*}
  \int_\I (\lnx)^- \ud x = & \int_{\I \cap (\lnx)^->0 } e^{(\lnx)^-} \ud x \leq   \int_{\I  } e^{(\lnx)^- - (\lnx)^+} \ud x\\
  = &\int_{\I  } e^{-(\ln h_x)_x} \ud x = \phi(h(t)) \leq \phi(h(0)),
\end{align*}
where $(\lnx)^-$ denotes the negative part of $\lnx$ and $(\lnx)^+ $ is the positive part of $\lnx$. Thanks to the periodic boundary condition, we have
\begin{equation}\label{1.16}
  \frac{1}{2} \|\lnx\|_{L^1(\I)} =\|(\lnx)^-\|_{L^1(\I)} = \|(\lnx)^+\|_{L^1(\I)} \leq \phi(h(0)).
\end{equation}
However, since $L^1$ is non-reflexive Banach space, the uniform bound of $L^1$ norm dose not prevent $\lnx$ being a Radon measure.
This gives us the idea to carry on all the calculations in BV space, i.e. $\ln h_x \in BV(\I);$ see explicit definition in Section 2.
\\
{\bf Outlines. } The rest of this paper is organized as follows. We will define functional $\phi$ and establish the gradient flow structure rigorously in Section 2.1, 2.2. Then after exploring some properties of $\phi$ in Section 2.3, we will prove the existence if EVI solution in Section 2.4. Section 3 is devoted to obtain the strong solution with latent singularity to \eqref{maineq}.
\section{Variational inequality solution}
\subsection{Preliminaries}
We first introduce the spaces we will work in. Notice the invariant property of \eqref{maineq} if we add a constant $c$ to solution $h$ and \eqref{Periodic}. Therefore without loss of generality, we consider $h$ with mean value zero.
Let
\begin{equation}\label{Hnote}
H:=\left\{u\in L^2(\I):\int_\I u\ud x=0\right\},
\end{equation}
endowed with the standard scalar product $\la u,v\ra _H:=\int_\I uv \ud x$. Here $u\in L^2(\mathbb{T})$ means $u$ satisfies the periodic boundary condition \eqref{Periodic}.

As in the observation 2, since $L^1$ is not reflexive Banach space and has no weak compactness, we work in a larger space, BV space. Denote $\mathcal{M}$ as the space of finite signed Radon measures and $\|\cdot\|_{\mathcal{M}(\I)}$ is the {total variation} of the measure. Define Banach space
\begin{equation}
  V:= \{u\in H; u_x \in BV(\I)\}.
\end{equation}
Endow $V$ with the norm
$$\|u\|_{V}:= \|u\|_{L^2(\I)}+ \|u_{xx}\|_{\mathcal{M}(\I)},$$
which is equivalent to the norm $\|u_{xx}\|_{\mathcal{M}(\I)}$ due to poinc\'are's inequality for mean value zero function.

Next, from observation 2 we expect $\ln h_x \in BV(\I)\hookrightarrow L^\infty(\I)$, which implies there will be a lower/upper bound for $h_x$. Therefore we expect there are constants $c_1, c_2$ such that $c_1\leq h_x \leq c_2$. Then the uniform estimate
$$\|\lnx\|_{\mathcal{M}(\I)}=\big\|\frac{h_{xx}}{h_x}\big\|_{\mathcal{M}(\I)}\leq \phi(h(0))$$
  will lead to a uniform bound for $\|h_{xx}\|_{\mathcal{M}(\I)}$. Since $h_{xx}$ can be a Radon measure, we need to make those formal observations  rigorous in Section \ref{sec1.2} by first defining $\phi$ properly.
Notice for any  $\mu\in \mathcal{M} ,$ from \cite[p.42]{evans1992}, we have the decomposition
\begin{equation}\label{decom}
  \mu=\mu_{\|} +\mu_{\bot}
\end{equation}
with respect to the Lebesgue measure,
where $\mu_{\|} \in L^1(\I)$ is the absolutely continuous part of $\mu$ and $\mu_{\bot}$ is the singular part, i.e., the support of $\mu_{\bot}$ has Lebesgue measure zero.
Define the beam type functional
\begin{equation}\label{phi}
\phi:H\to[0,+\8],\qquad \phi(h):=
\begin{cases}
\int_\I e^{-(\lnx)^+_{\|}+(\lnx)^- }\ud x , & \text{if } h\in V,\, \text{ and } (\lnx)^-\ll\mathcal{L}^1,\\
+\8 &\text{otherwise},
\end{cases}
\end{equation}
 Here $(\lnx)_{\|}$ denotes the absolutely continuous part of $\lnx$,  $(\lnx)^-$ is the negative part of $\lnx$ and $(\lnx)^+$ is the positive part of $\lnx$ such that $(\lnx)^{\pm} $ are two non-negative measures such that $\lnx= (\lnx)^+ -(\lnx)^-.$ We call the singular part  $(\lnx)^+_{\bot}$ latent singularity in solution $h$.

\medskip

In view of the a priori estimate on the mass of the measure $h_{xx}$,
 we introduce the indicator function
\begin{equation}\label{psi}
\psi:H\to\{0,+\8\},\qquad \psi(h):=
\begin{cases}
0 & \text{if } h\in V ,\ \|h_{xx}\|_{\mathcal{M}(\I)}\le C_*,\\
+\8 & \text{otherwise}.
\end{cases}
\end{equation}
Here $C_*$ is a fixed constant, which is determined in \eqref{consC} by the initial datum later.

\subsection{Euler Scheme}
Even if \eqref{maineq} has a nice variational structure, and $V$ has Banach space structure. To avoid the technical difficulties brought  by  non-reflexivity we adopt the result \cite[Theorem~4.0.4]{AGS}
by Ambrosio, Gigli and Savar\'e. After defining the energy functional rigorously, the key processes are to study the detail properties of energy functionals.
First let us we establish the gradient flow evolution in the {\em metric}
space $(H,\mbox{dist})$, with distance $\mbox{dist}(u,v):=\|u-v\|_H$.
Let $h^0\in {H}$ be a given initial datum and $0<\tau\ll1$ be a given parameter.
We consider a sequence $\{x_n^{\tau}\}$ which satisfies the following unconditional-stable backward Euler scheme
\begin{equation}\label{E}
\left\{
\begin{array}{l}
x^{(\tau)}_n\in \text{argmin}_{x'\in H} \left\{(\phi+\psi)(x')+\dfrac1{2\tau} \|x'-x^{(\tau)}_{n-1}\|^2_{H}
\right\}  \qquad n\ge 1,\\
x^{(\tau)}_0:= h^0\in H.
\end{array}
\right.
\end{equation}
The existence and uniqueness of the sequence $\{x_n^{\tau}\}$ can be proved by direct method in calculus of variation after we establishing the convexity and lower semi continuity of $\phi+\psi$ in Lemma \ref{newlsc}; see also \cite[Prop 11]{G2}.
Thus we are considering the gradient descent with respect to $\phi+\psi$ in the space $(H,\mbox{dist})$.

Now for any $0<\tau\ll1$ we define the resolvent operator (see \cite[p. 40]{AGS})
\begin{equation*}
\mathcal{J}_\tau[h]:=\text{argmin}_{v\in H} \left\{(\phi+\psi)(v)+\dfrac1{2\tau} \|v-h\|^2_{H}\right\},
\end{equation*}
then the variational approximation of $h$ at $t$ is obtained by Euler scheme \eqref{E} as
\begin{equation}\label{Euler10}
h_n(t):=(\mathcal{J}_{t/n})^n[h^0].
\end{equation}
In Proposition \ref{EVI}, we will use the theory for gradient flow in metric space \cite[Theorem~4.0.4]{AGS} to establish the convergence of the variational approximation $h_n(t)$ to variational inequality solution to \eqref{maineq}, which is defined below.
\begin{defn}\label{defweak}
Given initial data $h^0\in H$, we call $h:[0,+\8)\to H$ a variational inequality solution to \eqref{maineq} if $h(t)$ is a locally absolutely continuous curve such that $\lim_{t\to 0} h(t)=h^0$ in $H$ and
 \begin{equation}
\la h_t(t),h(t)-v\ra_{{ H',H}}\le \phi(v)-\phi(h(t)) \quad \text{for a.e. } t>0,\,\forall v\in { D(\phi+\psi)}.
\end{equation}
\end{defn}
Next we study some properties like convexity and lower semi continuity in $H$, of the functional $\phi+\psi$.

\subsection{Convexity and lower semi continuity of function $\phi+\psi$ in $H$}
Before we prove the convexity and lower semi continuity of function $\phi+\psi$, we first state an important lemma concerning the weak lower semi continuity of $\phi$ in BV space.
\begin{prop}\label{newlsc}
  Let $h_n,\, h\in V$. If $(\ln h_{nx})_x {\rhu} \lnx$ in $\mathcal{M}(\I)$, we have
  \begin{equation}\label{lsc01}
  \liminf_{n\to +\8}\phi(h_n)\geq \phi(h).
  \end{equation}
\end{prop}
\begin{proof}
 Denote $\mu_{n}:=(\ln h_{nx})_x$, $\mu:=\lnx$.
Notice that $\phi$ defines only on the absolutely continuous part of $\lnx$. Hence the key point is to clarify the cases (1) the absolutely continuous part of $\mu_n$ converge to the singular part of $\mu$; and (2) the singular part of $\mu_n$ converge to  the absolutely continuous part of $\mu$. We refer to \cite[Proposition 5]{G2} for the proof of these two cases.
\end{proof}

Next we will prove the convexity and lower semi continuity of function $\phi+\psi$ in $H$.
\begin{lem}\label{convex}
The sum $\phi+\psi:H\to [0,+\8]$ is proper, convex, lower semicontinuous in $H$ and satisfies coercivity defined in \cite[(2.4.10)]{AGS}.
%{\blue Q2: Shall we also state the coercivity of $\phi+\psi$ as (4.0.1) in \cite{AGS}[Theorem 4.0.4], but it is obviously coercive, right?}
\end{lem}

\begin{proof}
Clearly since the typical function $h= Lx\in D(\phi+\psi)$, so $D(\phi+\psi)=\{\phi+\psi<+\8\}$ is non empty and $\phi+\psi$ is proper.
Due to the positivity of $\phi,\, \psi$, coercivity \cite[(2.4.10)]{AGS}, {i.e.,
$\exists u*\in D(\phi+\psi), r*>0 \text{ such that } \inf\{(\phi+\psi) (v): v\in H, \text{dist}(v,u*)\leq r*\}>-\infty,$
} is obvious.
\medskip

{\em Convexity.} Note that since both $\phi$, $\psi\ge 0$, we have $D(\phi+\psi)=D(\phi)\cap D(\psi)$.
Given $u,v\in H$, $t\in (0,1)$, without loss of generality assume $u,v\in D(\phi+\psi)$,
otherwise convexity inequality is trivial. Therefore
the measure $(1-t)(\ln u_x)_x+ t (\ln v_x)_x$ has no negative singular part, while its positive
singular part satisfies
$$[(1-t)(\ln u_x)_x+ t (\ln v_x)_x]^+_\bot=[(1-t)(\ln u_x)_x]_\bot^+ +[t (\ln v_x)_x]_\bot^+,$$
and its absolutely continuous part satisfies
$$[(1-t)(\ln u_x)_x+ t (\ln v_x)_x]_\|=[(1-t)(\ln u_x)_x]_\| +[t (\ln v_x)_x]_\|.$$
Thus we have
\begin{align*}
\phi((1-t)u+tv) & =\int_\I e^{-\Big[\big(\ln [(1-t)u_x+tv_x]\big)_x\Big]_\|}\ud x\\
&\le \int_\I e^{-\big[(1-t)(\ln u_x)_x+t (\ln v_x)_x \big]_\|}\ud x\\
&= \int_\I e^{-(1-t)\big[(\ln u_x)_x\big]_\|-t \big[(\ln v_x)_x \big]_\|}\ud x\\
&\leq (1-t)\int_\I e^{-\big[(\ln u_x)_x\big]_\|}\ud x+t \int_\I e^{- \big[(\ln v_x)_x \big]_\|}\ud x\\
&= (1-t)\phi(u)+t\phi(v),
\end{align*}
where we used the convexity of $-\ln x$ and $e^{-x}$ in the two inequalities separately.
Hence $\phi+\psi$ is convex.

\medskip

{\em Lower semicontinuity.} Consider a sequence $h_n\to h$ in $H$. We need to check
$$(\phi+\psi)(h)\le \liminf_n (\phi+\psi)(h_n).$$
If $h_n\in D(\phi+\psi)$ does not hold
for all large $n$, then lower semicontinuity holds.
Without loss of generality, we can assume $h_n\in D(\phi+\psi)$ for all $n$, and also
$$\liminf_n (\phi+\psi)(h_n)=\lim_n (\phi+\psi)(h_n).$$

First notice $h_n\in D(\phi)$ for any $n$ implies
\begin{align*}
  \int_\I (\lnxn)^- \ud x = & \int_{\I \cap (\lnxn)^->0 } e^{(\lnxn)^-} \ud x \leq   \int_{\I  } e^{(\lnxn)^- - (\lnxn)^+_{\|}} \ud x\\
  = &\phi(h_n(t)) \leq C.
\end{align*}
Then similar to \eqref{1.16}, we have
\begin{equation}
  \frac{1}{2} \|\lnxn\|_{\mathcal{M}(\I)} =\|(\lnxn)^-\|_{\mathcal{M}(\I)} = \|(\lnxn)^+\|_{\mathcal{M}(\I)} \leq C,
\end{equation}
which yields that there exists $\mu\in \mathcal{M}(\I)$ such that $\lnxn \rhu \mu$ in $\mathcal{M}(\I).$

Second, since $h_n\in D(\psi)$, we have $\|h_{nxx}\|_{\mathcal{M}(\I)}\le C_*$. Thus strong convergence $h_n\to h$ in $H$, together with the fact $BV(\I)\hookrightarrow L^p(\I)$ compactly for any $p<\8$, leads to the strong convergence $h_{nx}\to h_x$ in $L^p(\I)$ for any $p<\8$. Therefore we have $h_{nx}\to h_x$ almost everywhere and consequently $\ln h_{nx}\to \ln h_x$ almost everywhere. Combining this with $\lnxn \rhu \mu$ in $\mathcal{M}(\I)$ gives $\mu=\lnx$ and $\lnxn \rhu \lnx$ in $\mathcal{M}(\I)$.

Finally, since $h_{nxx}\rhu h_{xx}$ in $\mathcal{M}(\I)$, we also know $h\in D(\psi)$ and
 $0=\psi(h_n)=\psi(h)$. Therefore by Proposition \ref{newlsc} we have
\begin{equation*}
\liminf_n \phi(h_n) \geq \phi(h)
\end{equation*}
and the lower semicontinuity is proved.
\end{proof}

As long as we have the convexity of $\phi+\psi$, the $\tau^{-1}$-convexity is standard and the proof can be found in \cite[Lemma 10]{G2}.
\begin{prop}[$\tau^{-1}$-convexity]\label{convex2}
For any $h,v_0,v_1\in D(\phi+\psi)$, there exists a curve $v:[0,1]\to D(\phi+\psi)$
such that $v(0)=v_0,\, v(1)=v_1$ and
the functional
\begin{equation}\label{PhiD}
\Phi(\tau,h;v):=(\phi+\psi)(v)+\frac1{2\tau} \|h-v\|_H^2
\end{equation}
satisfies $\tau^{-1}$-convexity, i.e.,
\begin{equation}\label{c}
\Phi(\tau,h;v(t))\le (1-t)\Phi(\tau,h;v_0)+t\Phi(\tau,h;v_1)-\frac1{2\tau}t(1-t)\|v_0-v_1\|_H^2
\end{equation}
for all $\tau>0$, $t\in[0,1]$.
\end{prop}

\subsection{Existence of variational inequality solution}\label{Sec2.4}
After studying convexity and lower semicontinuity in last section, we shall apply the convergence result in \cite[Theorem 4.0.4]{AGS} to derive that the discrete solution $h_n$ obtained by Euler scheme \eqref{E} converges to the variational inequality solution defined in Definition \ref{defweak}. For $v\in D(\phi)$, denote the local slope
\begin{equation}\label{localslope}
 |\pd \phi|(v):=\limsup_{w\to v}\frac{\max\{\phi(v)-\phi(w),0\}}{\mbox{dist}(v,w)}.
\end{equation}
\begin{prop}\label{EVI}
  Given $h^0 \in H$, for any $t>0$, $t=n\tau$, let $h_n(t)$ defined in \eqref{Euler10} be the approximation solution obtained by Euler scheme \eqref{E}, then there exists a local Lipschitz curve $h(t):[0,+\8)\to H$ such that
      \begin{equation}\label{tmp23}
        h_n(t)\to h(t)\text{  in }L^2(\I)
      \end{equation}
      and $h:[0,+\8)\to H$ is the unique EVI solution in the sense that $h$ is unique among all the locally absolutely continuous curves such that $\lim_{t\to 0} h(t)=h^0$ in $H$ and
\begin{equation}\label{vi}
\frac12\frac\ud{\ud t}\|h(t)-v\|^2\le (\phi+\psi)(v)-(\phi+\psi)(h(t)), \quad \text{  a.e. } t>0,\,\forall v\in D(\phi+\psi).
\end{equation}
Moreover, we have the following regularities
\begin{align}
(\phi+\psi)(h(t)) &\le (\phi+\psi)(v)+\frac1{2t}\|v-h^0\|_H^2, \qquad \forall v\in D(\phi+\psi),\label{phi-dec}\\
|\pd(\phi+\psi)|^2(h(t)) &\le |\pd(\phi+\psi)|^2(v)+\frac1{t^2}\|v-h^0\|_H^2,\qquad\forall v\in D(|\pd(\phi+\psi)|).\label{Dphi-dec}
\end{align}
\end{prop}
This Proposition is a direct result by combining \cite[Theorem 4.0.4]{AGS} with Proposition \ref{newlsc} and Proposition \ref{convex2}. Other regularities estimates can also be obtain and we refer to \cite[Theorem 13]{G2}, \cite[Theorem 4.0.4]{AGS} for details. Next we claim the EVI solution obtained above is EVI solution to \eqref{maineq} with more properties as follows.
\begin{thm}\label{cor10}
Given any  $T>0$ and initial datum $h^0 \in H$ such that $\phi(h^0)<+\8$,
\begin{enumerate}[(i)]
\item  the solution obtained in Proposition \ref{EVI} has the following regularities
 $$h\in L^\8([0,T];V )\cap C^0([0,T];H),\quad h_t\in L^\8([0,T];H),$$
 $$(\lnx)^-\ll\mathcal{L}^1\quad \text{ for a.e. }t\in[0,T],$$
 where
 $(\lnx)^-$ is the negative part of $\lnx$;
 \item there exist constants $c_1, c_2>0$ depending only on $h_0$ and will be determined in \eqref{c12} such that
 \begin{equation}\label{lowup}
   c_1\leq h_x \leq c_2;
 \end{equation}
 \item $h$ is the EVI solution in Definition \ref{defweak}, i.e.
 \begin{equation}\label{vi35}
\la h_t(t),h(t)-v\ra_{{ H',H}}\le \phi(v)-\phi(h(t)) \quad \text{for a.e. } t>0,\,\forall v\in { D(\phi+\psi)},
\end{equation}
and consequently we have the decay estimate
\begin{equation}\label{decay1}
 \frac{\ud}{\ud t} \int_\I h^2 \ud x \leq 0.
\end{equation}
\end{enumerate}
\end{thm}
The dual pair $\la \cdot, \cdot \ra_{H',H}$ is the usual integration so we just use $\la \cdot, \cdot \ra$ in the following article.
 Recall the definition of $\phi$ in \eqref{phi}. $\phi(h^0)<+\8$ if and only if $h^0\in V$, $((\ln h_x^0)_x)^-\ll\mathcal{L}^d$ and $\int_\I e^{-((\ln h_x^0)_x)^+_\|+((\ln h_x^0)_x)^-}\ud x<+\8.$
\begin{proof}
 First,
we claim the functional $\psi$ can be taken off. Indeed, from \eqref{phi-dec}  taking $v=h^0$ gives
\begin{equation}\label{phi0}
(\phi+\psi)(h(t))\le (\phi+\psi)(h^0)<+\8,
\end{equation}
which also implies
\begin{equation}\label{phi0n}
\phi(h(t))\le \phi(h^0)<+\8 \quad \text{  for a.e. }t\in[0,T].
\end{equation}
Now we use \eqref{phi0n} to determine those constants $c_1, c_2$ in Theorem \ref{cor10} and $C_*$ in Definition \ref{psi}. Notice the periodic boundary condition
 we have  $\int_\I \ud (h_{xx})=0$, and then
\begin{equation}\label{pmbn}
\|(h_{xx})^+\|_{\mathcal{M}(\I)}=\|(h_{xx})^-\|_{\mathcal{M}(\I)}=\frac12\|h_{xx}\|_{\mathcal{M}(\I)}.
\end{equation}
Thanks to
\begin{align*}
  \int_\I (\lnx)^- \ud x = & \int_{\I \cap (\lnx)^->0 } e^{(\lnx)^-} \ud x \leq   \int_{\I  } e^{(\lnx)^- - (\lnx)^+_{\|}} \ud x\\
  = &\phi(h(t)) \leq \phi(h^0),
\end{align*}
we know $(\lnx)^-\ll\mathcal{L}^1 \text{ for a.e. }t\in[0,T].$
Then similar to \eqref{pmbn}, we have
\begin{equation}\label{lnbound}
  \frac{1}{2} \|\lnx\|_{\mathcal{M}(\I)} =\|(\lnx)^-\|_{\mathcal{M}(\I)} = \|(\lnx)^+\|_{\mathcal{M}(\I)} \leq \phi(h^0).
\end{equation}
Due to the embedding $BV(\I)\hookrightarrow L^\8(\I)$ in one dimension, we have
$$\|\ln h_x\|_{L^\8(\I)}\leq c \phi(h^0),$$
which implies
\begin{equation}\label{c12}
  c_1:=e^{-c\phi(h^0)} \leq h_x \leq e^{c\phi(h^0)}=:c_2
\end{equation}
and (ii).
Combining \eqref{lnbound} and \eqref{c12}, we conclude
\begin{equation}\label{tm2.26}
  \frac{1}{c_1}\|h_{xx}\|_{\mathcal{M}(\I)}\leq \| \frac{h_{xx}}{h_x} \| _{\mathcal{M}(\I)}\leq 2 \phi(h^0).
\end{equation}
Therefore in Definition \eqref{psi}, we can just take
\begin{equation}\label{consC}
 C_*:=2c_1\phi(h^0)+1
\end{equation}
 and then
\begin{equation}\label{psi0}
\psi(h(t))\equiv 0\equiv \partial\psi(h(t)).
\end{equation}
The invariant ball introduced by indicate function $\psi$ is similar to the idea of a priori assumption method, i.e. we first obtain the solution in some invariant ball $\|h_{xx}\|_{\mathcal{M}}\leq C_*$, and then prove the invariant ball is not artificial by showing the solution always locates within the ball $\|h_{xx}\|_{\mathcal{M}}\leq C_*-1.$
Noticing also that if $v\in D(\psi)$, $\psi(v)=0$, so
 EVI \eqref{vi} is reduced to
$$\frac12\frac\ud{\ud t}\|h(t)-v\|^2\le \phi(v)-\phi(h(t)), \quad \text{for a.e. } t>0,\,\forall v\in D(\phi+\psi).
$$

Second, it remains to proof the $h_t\in L^\8(0,T;L^2(\I))$. From Theorem \ref{EVI} we know that $t\mapsto h(t)$ is locally Lipschitz in
$(0,T)$, i.e. for any $t_0>0$ there exists $M=M(t_0)>0$ such that
$$\|h(t_0+\eps)-h(t_0)\|_{L^2(\I)}\le M(t_0)\eps \qquad \text{for all } {\eps\in [0, T-t_0]}.$$
The key point is to obtain a uniform bound for $M(t_0)$ for arbitrary $t_0\ge 0$.  By exactly the same argument in \cite[Corollary3.1]{G2} we can show
 \begin{equation}\label{ut925}
 \|h_t\|_{L^\8(0,T;L^2(\I))}\le |\pd \phi|(h^0),
 \end{equation}
 which concludes (i).
 
 Finally, from
$$\frac12\frac{\ud}{\ud t}\|h(t)-v\|_{L^2(\I)}^2=\la h_t(t), h(t)-v\ra,$$
 we obtain \eqref{vi35}. From \eqref{vi35}, substituting $v=(1-\eps)h$ for any $\eps$ small enough shows
 $$\la h_t(t), \eps h(t) \ra \leq \phi ((1-\eps)h)-\phi(h(t))=0,$$
  which concludes \eqref{decay1}.
\end{proof}

\section{Existence of strong solution}
After establishing the regularity of variational inequality solution in Section \ref{Sec2.4}, we start to prove the variational inequality solution is also a strong solution. We first clarify the definition of strong solution, which has a latent singularity. We see from observation 2 that the singularity in positive part of $(\ln h_x)_x$ does not effect the evolution of the solution but it is not removable.  Singular PDEs should be understood as a limit of some regularized problems. In our case, the singular PDE is understood in the sense that the equation holds almost everywhere after removing  the singular part of $(\ln h_x)_x$.

\begin{defn}\label{defstrong}
Given initial datum $h^0 \in H$ such that $\phi(h^0)<+\8$, we call function
 $$h\in L^\8([0,T];V)\cap C^0([0,T];H),\quad h_t\in L^\8([0,T];H)$$
 a strong solution to \eqref{maineq} if $h$ satisfies
 \begin{equation}\label{tt57}
         h_t=\left( \frac{1}{h_x} \left(e^{-((\ln h_x)_x)_\|} \right)_x \right)_x
        \end{equation}
        for a.e. $(t,h)\in[0,T]\times \I$ with respect to Lebesgue measure, where $((\ln h_x)_x)_\|$ is the absolutely continuous part of $(\ln h_x)_x)$ in the decomposition \eqref{decom}.
\end{defn}
The idea is to prove the sub-differential of functional $\phi$ is single-valued by  testing EVI \eqref{vi35} with $v:= h\pm \eps \varphi$ for any function $\varphi \in C^\8 (\I).$
Let us state the main existence  theorem as follows.
\begin{thm}\label{mainth1}
  Given $T>0$, initial datum $h^0 \in H$ such that $\phi(h^0)<+\8$, then { EVI} solution $h$ obtained in Theorem \ref{cor10} is  also a strong solution
  to \eqref{maineq}, i.e.,
        \begin{equation}\label{main21_3}
                h_t=\left( \frac{1}{h_x} \left(e^{-((\ln h_x)_x)_\|} \right)_x \right)_x
        \end{equation}
        for a.e. $(t,h)\in[0,T]\times \I$ with respect to Lebesgue measure.
        Besides, we have
        the following dissipation inequality
        \begin{equation}\label{tem-dis1}
         \phi(h(t))=\int_\I e^{-((\ln h_x)_x)_\|} \ud x\leq \phi(h^0), \quad t\geq 0,
        \end{equation}
%        Furthermore, if $E(h^0):=\frac{1}{2}\int_\I \big[\left( \frac{1}{h_x} \left(e^{-((\ln h_x)_x)_\|} \right)_x \right)_x\big]^2 \ud x<\8,$ then
%         \begin{equation}\label{dissiE}
%           E(h(t)):=\frac{1}{2}\int_\I \big[\left( \frac{1}{h_x} \left(e^{-((\ln h_x)_x)_\|} \right)_x \right)_x\big]^2 \ud x\leq E(h^0), \quad t\geq 0,
%         \end{equation}
        where $((\ln h_x)_x)_\|$ is the absolutely continuous part of $(\ln h_x)_x)$ in the decomposition \eqref{decom}.
\end{thm}
\begin{proof}
The general idea is to character the sub-differential of functional $\phi$  by  testing EVI \eqref{vi35} with $v:= h\pm \eps \varphi$ for any function $\varphi \in C^\8 (\I).$

  Step 1. Integrability results.
  
   Assume $h(t)$ is EVI solution obtained in Theorem \ref{cor10}. To ensure we can take limit after testing EVI with $v:= h\pm \eps \varphi$, we need to prove
   \begin{equation}\label{es1n}
     e^{-((\ln h_x)_x)_\|} \in L^1(\I)
   \end{equation}
   and for $\eps$ small enough
   \begin{equation}\label{es2n}
     e^{-((\ln (h_x+\eps \varphi_x))_x)_\|} \in L^1(\I)
   \end{equation}
for $\varphi$ in some dense set of $C_b^\8(\I).$
   First from \eqref{phi0n} we know $\phi(h(t))\leq \phi(h^0)$, which gives \eqref{tem-dis1} and \eqref{es1n}.

   Next, we prove \eqref{es2n} for $\varphi$ in some dense set of $C_b^\8(\I)$.
   For any $c>0$, define
   \begin{equation}
     D_c := \{ \varphi\in C_b^\8(\I); |(h_{xx})_\| \varphi_x |\leq c \}; \quad D:= \cup_{c\geq 0} D_c.
   \end{equation}
   We claim the set $D$ is dense in $L^\8(\I)$. Indeed, for any $\varphi\in L^\8(\I)$ define
   $$
   \varphi_n:= \left\{
                 \begin{array}{ll}
                   \varphi & \hbox{ if } |(h_{xx})_\| \varphi_x | \leq n; \\
                   0 & \hbox{ otherwise.}
                 \end{array}
               \right.
   $$
Then for any $1\leq p< \8$,
$$
\|\varphi_n-\varphi\|_{L^p(\I)}= \left( \int_{\{|(h_{xx})_\| \varphi_x|>n \}} |\varphi_n -\varphi|^p  \ud x \right)^{\frac{1}{p}}.
$$
Since $|(h_{xx})_\| \varphi_x|\leq |(h_{xx})_\| | \|\varphi_x\|_{L^\8}$,
$$ \{|(h_{xx})_\| \varphi_x|>n\} \subseteq \{|(h_{xx})_\| | \|\varphi_x\|_{L^\8}>n\}. $$
Therefore
\begin{align*}
  \|\varphi_n-\varphi\|_{L^p(\I)} & \leq  \left( \int_{\{|(h_{xx})_\| |>\frac{n}{\|\varphi_x\|_{L^\8}} \}} |\varphi_n -\varphi|^p  \ud x \right)^{\frac{1}{p}}\\
& \leq \|\varphi\|_{L^\8} \Big| \big\{|(h_{xx})_\| |>\frac{n}{\|\varphi_x\|_{L^\8}} \big\} \Big|^\frac{1}{p} \to 0
\end{align*}
as $n\to \8$, where we used the integrability $(h_{xx})_\| \in L^1 (\I).$
Then we know $D$ is dense in $L^\8$ and thus $D$ is dense in $C_b^\8(\I).$

For any $\varphi\in D$, $|\varphi_{xx}|\leq \|\varphi_{xx}\|_{L^\8}$ and there exists some $c$ such that
$|(h_{xx})_\| \varphi_x|\leq c$. Notice also $c_1\leq h_x \leq c_2$ due to \eqref{lowup}.
Hence for $\eps$ small enough,
\begin{align*}
  &\int_\I e^{  -\big((\ln (h_x+\eps \varphi_x))_x\big)_\| } \ud x = \int_\I e^{  -\Big[ \frac{h_{xx}+\eps \varphi_{xx} }{ h_x + \eps \varphi_x } \Big]_\| } \ud x\\
 = &\int_\I e^{ - \frac{(h_{xx})_\|}{ h_x + \eps \varphi_x }  } e^{ - \eps\frac{\varphi_{xx}}{ h_x + \eps \varphi_x }  }  \ud x\\
\leq & C(c_1, c_2, \|\varphi\|_{W^{2,\8}})  \left[\int_{((h_{xx})_\|)^+>0} e^{ - \frac{(h_{xx})_\|}{ h_x + \eps \varphi_x}  } \ud x + \int_{((h_{xx})_\|)^+=0} e^{ - \frac{(h_{xx})_\|}{ h_x + \eps \varphi_x }  } \ud x \right]\\
 \leq & C(c_1, c_2, \|\varphi\|_{W^{2,\8}})  + C( c_1, c_2, \|\varphi\|_{W^{2,\8}})  \int_{((h_{xx})_\|)^+=0} e^{ - \frac{(h_{xx})_\|}{h_x} (1 + \eps \frac{|\varphi_x|}{h_x} ) } \ud x.
 \end{align*}
 Here for the second term in the last inequality, we used when $((h_{xx})_\|)^+=0$,
$$-\frac{(h_{xx})_\|}{h_x} \frac{1}{1+\eps\frac{\varphi_x}{h_x}}\leq -\frac{(h_{xx})_\|}{h_x} \big( 1+ 2 \eps\frac{|\varphi_x|}{h_x}\big)$$ 
due to $\frac{1}{1+\eps y} \leq 1+ 2\eps |y|$ for any $\eps<\frac{1}{2\max|y|}.$
 Therefore
 \begin{align*}
 &\int_\I e^{  -\big((\ln (h_x+\eps \varphi_x))_x\big)_\| } \ud x \\
 \leq & C(c_1, c_2, \|\varphi\|_{W^{2,\8}})  + C( c_1, c_2, \|\varphi\|_{W^{2,\8}})  \int_{((h_{xx})_\|)^+=0} e^{ - \frac{(h_{xx})_\|}{h_x} (1 + 2\eps \frac{|\varphi_x|}{h_x} ) } \ud x\\
\leq & C(c_1, c_2, \|\varphi\|_{W^{2,\8}})  + C( c_1, c_2, \|\varphi\|_{W^{2,\8}})  \int_{((h_{xx})_\|)^+=0} e^{ - \frac{(h_{xx})_\|}{h_x} } e^{ -2\eps \frac{(h_{xx})_\| |\varphi_x|}{h_x^2}  } \ud x\\
\leq & C(c_1, c_2, \|\varphi\|_{W^{2,\8}})  + C(c, c_1, c_2, \|\varphi\|_{W^{2,\8}}) \int_{\I} e^{ - \frac{(h_{xx})_\|}{h_x} } \ud x,
\end{align*}
where we used $|(h_{xx})_\| \varphi_x|\leq c$ in the last inequality and $C(c_1, c_2, \|\varphi\|_{W^{2,\8}})$ is a general constant depending only on $c_1, c_2, \|\varphi\|_{W^{2,\8}}$. This, together with \eqref{es1n} leads to \eqref{es2n}.

Step 2. Testing \eqref{vi35} with $v=h\pm \eps \varphi.$

First we show $v\in D(\phi+\psi)$. Since $\varphi \in C_b^\8(\I)$ and \eqref{lowup}, we can choose $\eps$ small enough such that $h_x+ \eps \varphi_x >0$, so $(h_{xx}+\eps \varphi_{xx})^- \in L^1$ and $v\in D(\phi)$.
 It is sufficient to show $v\in D(\psi)$ for $\eps$ small enough. Indeed, from \eqref{tm2.26} we know $\|h_{xx}\|_{\mathcal{M}}\leq 2c_1\phi(h^0)=C_*-1.$ Hence we choose $\eps$ small enough such that $\eps\leq \frac{1}{2\|\varphi\|_{W^{2,\8}}}$, which implies $\|v\|_{\mathcal{M}}\leq 2c_1\phi(u^0)+\frac{1}{2}< C_*$ and $\psi(v)=0.$

Plugging $v=h+ \eps \varphi$ into \eqref{vi35} gives
\begin{align}\label{tempeps}
0\leq & \la h_t(t),\eps\varphi\ra+\phi(h(t)+\eps \varphi)-\phi(h(t))\nonumber \\
=& \eps\la h_t(t),\varphi\ra + \int_{\I} e^{ -\Big[ \frac{h_{xx}+\eps \varphi_{xx} }{ h_x + \eps \varphi_x } \Big]_\| } - e^{ - \frac{(h_{xx})_\|}{h_x}  } \ud x.
\end{align}
We divide this by $\eps>0$ and take limit $\eps\to 0^+$. Thanks to the dominated convergence theorem and the integrability \eqref{es2n}, we  just need to check the pointwise limit for the integrand in the dense set $D$. For any $x\in \I$, $\varphi\in D$. we have
\begin{align*}
  &\frac{1}{\eps} \left[ e^{ -\Big[ \frac{h_{xx}+\eps \varphi_{xx} }{ h_x + \eps \varphi_x } \Big]_\| } - e^{ - \frac{(h_{xx})_\|}{h_x}  } \right]\\
=& \frac{1}{\eps} \left[  e^{ -\frac{ (h_{xx})_\| + \eps \varphi_{xx} }{h_x (1+ \eps \frac{\varphi_x}{h_x}) }  } - e^{ - \frac{(h_{xx})_\|}{h_x}  } \right]\\
= &\frac{1}{\eps}  \left[  e^{ -\frac{ (h_{xx})_\| + \eps \varphi_{xx} }{h_x}  [1- \eps \frac{\varphi_x}{h_x} + O(\eps^2)]}   - e^{ - \frac{(h_{xx})_\|}{h_x}  } \right]\\
= & \frac{1}{\eps} \left[  e^{  -\frac{(h_{xx})_\|}{h_x}- \eps \frac{\varphi_{xx}}{h_x} + \eps \frac{(h_{xx})_\|}{h_x}\frac{\varphi_x}{h_x} + O(\eps^2) } - e^{ - \frac{(h_{xx})_\|}{h_x}  } \right]\\
= & \frac{1}{\eps}  e^{  -\frac{(h_{xx})_\|}{h_x}- \eps \frac{\varphi_{xx}}{h_x} + \eps \frac{(h_{xx})_\|}{h_x}\frac{\varphi_x}{h_x} + O(\eps^2) } \big[ 1- e^{  \eps \frac{\varphi_{xx}}{h_x} - \eps \frac{(h_{xx})_\|}{h_x}\frac{\varphi_x}{h_x} + O(\eps^2) } \big] \\
\leq &   \frac{1}{\eps}  e^{  -\frac{(h_{xx})_\|}{h_x}- \eps \frac{\varphi_{xx}}{h_x} + \eps \frac{(h_{xx})_\|}{h_x}\frac{\varphi_x}{h_x} + O(\eps^2) } \big[  - \eps \frac{\varphi_{xx}}{h_x} + \eps \frac{(h_{xx})_\|}{h_x}\frac{\varphi_x}{h_x} + O(\eps^2) \big]\\
\to & e^{  -\frac{(h_{xx})_\|}{h_x}} \Big[\  -  \frac{\varphi_{xx}}{h_x} +  \frac{(h_{xx})_\|}{h_x}\frac{\varphi_x}{h_x}  \Big]
\end{align*}
as $\eps\to 0^+ $, where we used $1-e^{x}\leq -x$ for all $x\in \mathbb{R}$ in the inequality.
Then taking limit in \eqref{tempeps} yields
\begin{align*}
&  \la h_t(t),\varphi\ra + \lim_{\eps \to 0} \frac{1}{\eps} \int_\I \left[ e^{ -\Big[ \frac{h_{xx}+\eps \varphi_{xx} }{ h_x + \eps \varphi_x } \Big]_\| } - e^{ - \frac{(h_{xx})_\|}{h_x}  } \right] \ud x\\
=&\la h_t(t),\varphi\ra + \int_\I  e^{  -\frac{(h_{xx})_\|}{h_x}} \Big[\  -  \frac{\varphi_{xx}}{h_x} +  \frac{(h_{xx})_\|}{h_x}\frac{\varphi_x}{h_x}  \Big] \ud x \geq 0
\end{align*}
for any $\varphi\in D$. Repeating the above arguments with $v=h(t)-\eps \varphi$ gives
$$\la h_t(t),\varphi\ra + \int_\I  e^{  -\frac{(h_{xx})_\|}{h_x}} \Big[\  -  \frac{\varphi_{xx}}{h_x} +  \frac{(h_{xx})_\|}{h_x}\frac{\varphi_x}{h_x}  \Big] \ud x \leq 0.$$
Then we finally obtain
\begin{equation}
   \int_\I h_t(t) \varphi+  e^{  -\frac{(h_{xx})_\|}{h_x}} \Big[\  -  \frac{\varphi_{xx}}{h_x} +  \frac{(h_{xx})_\|}{h_x}\frac{\varphi_x}{h_x}  \Big] \ud x = 0.
\end{equation}
By the dense argument for G\^{a}teaux-derivative, this equality  holds for any $\varphi\in C_b^\8(\I)$.

{Now we integrate by parts  for $\int_\I -\frac{1}{h_x}e^{  -\frac{(h_{xx})_\|}{h_x}} \varphi_{xx} \ud x $ and 
by the Radon-Nikodym theorem,
\begin{align*}
 \int_\I h_t(t) \varphi+  \left[  \frac{1}{h_x}\left[ \big( e^{-\frac{(h_{xx})_\|}{h_x}} \big)_x \right]_\| + \frac{1}{h_x}\left[ \big( e^{-\frac{(h_{xx})_\|}{h_x}} \big)_x \right]_{\bot} -\frac{(h_{xx})_{\bot} }{h_x^2} e^{-\frac{(h_{xx})_\|}{h_x}} \right] \varphi_x \ud x =0.
\end{align*}
Therefore
$$ h_t -  \left[  \frac{1}{h_x}\left[ \big( e^{-\frac{(h_{xx})_\|}{h_x}} \big)_x \right]_\| + \frac{1}{h_x}\left[ \big( e^{-\frac{(h_{xx})_\|}{h_x}} \big)_x \right]_{\bot} -\frac{(h_{xx})_{\bot} }{h_x^2} e^{-\frac{(h_{xx})_\|}{h_x}} \right]_x =0$$
in $(C_b^\8(\I))'$, which leads to
$$h_t -   \left[\frac{1}{h_x} \big( e^{-\frac{(h_{xx})_\|}{h_x}} \big)_x \right]_x=0$$
for a.e. $(t, x)\in [0,T]\times \I$ with respect to Lebesgue measure and concludes $h$ is a strong solution.
}\end{proof}

\bigskip

%-----------------------------------------------------------------------------bibliography


\begin{thebibliography}{00}

\bibitem{She2011} H. Al Hajj Shehadeh, R. V. Kohn and J. Weare, The
  evolution of a crystal surface: Analysis of a one-dimensional step
  train connecting two facets in the adl regime, Physica D: Nonlinear
  Phenomena 240 (2011), no. 21, 1771--1784.

\bibitem{AGS}
L. Ambrosio, N. Gigli and G. Savar\'e, {\em Gradient Flows
in Metric Spaces and in the Space
of Probability Measures}, Birkh\"auser Verlag, Basel, 2008

%\bibitem{Barbu2010}
%V. Barbu, Nonlinear differential equations of monotone types in banach spaces, Springer, New York, 2010.

%\bibitem{BrezisF}
%H. Brezis, {\em Functional Analysis, Sobolev Spaces and Partial Differential Equations}, Springer New York, 2010.

%\bibitem{Brezis1973}
%H. Brezis,
%Op{\'e}rateurs maximaux monotones et semi-groupes de contractions dans les espaces de Hilbert,
%North-Holland, Elsevier, 1973.

\bibitem{BCF} W. K. Burton, N. Cabrera and F. C. Frank, The growth of
  crystals and the equilibrium structure of their surfaces,
  Philosophical Transactions of the Royal Society of London A:
  Mathematical, Physical and Engineering Sciences 243 (1951), no. 866,
  299--358.


%\bibitem{DeTe1984}
%F. Demengel and R. Temam, Convex functions of a measure and applications.
%Indiana Univ. Math. J. 33 (1984), 673-709.

\bibitem{Yip2001} W. E and N. K. Yip, Continuum theory of epitaxial
  crystal growth. I, Journal Statistical Physics 104 (2001), no. 1-2,
  221--253.

  \bibitem{evans1992}
{\sc L.C. Evans, R.F. Gariepy,} { Measure Theory and Fine Properties of Functions}, CRC Press, 1992.




%\bibitem{Brezis1972} H. Br\'ezis, Probl\`emes unilat\'eraux,
 % J. Math. Pures Appl. 51 (1972) 1--168.
%\bibitem{Friedman1990} F. Bernis and A. Friedman, Higher order nonlinear degenerate parabolic equations, Journal of Differential Equations 83 (1990), no. 1, 179-206.



%\bibitem{Butzer1971} P. L. Butzer and R. J. Nessel, Fourier analysis and approximation, vol. 40, Academic Press, 2011.


%\bibitem{Leoni2014} G. Dal Maso, I. Fonseca and G. Leoni, Analytical validation of a continuum model for epitaxial growth with elasticity on vicinal surfaces, Archive for Rational Mechanics and Analysis 212 (2014), no. 3, 1037--1064.

%\bibitem{Duport1995a} C. Duport, P. Nozieres, and J. Villain, New
%  instability in molecular beam epitaxy, Phys.  Rev. Lett., 74 (1995),
%  pp. 134--137.

%\bibitem{Duport1995b} C. Duport, P. Politi, and J. Villain, Growth instabilities induced by elasticity in a vicinal
%surface, J. Phys. I France, 5 (1995), pp. 1317--1350.




%\bibitem{Evans1998} L. C. Evans, Partial Differential Equations (Graduate Studies in Mathematics vol 19), Providence, RI: American Mathematical Society, 1998.

\bibitem{Leoni2015} I. Fonseca, G. Leoni and X. Y. Lu, Regularity in time for weak solutions of a continuum model for epitaxial growth with elasticity on vicinal surfaces, Communications in Partial Differential Equations 40 (2015), no. 10, 1942--1957.

%\bibitem{FunakiSpohn:97}
%T. Funaki and H. Spohn, Motion by Mean Curvature from the Ginzburg-Landau Interface Model, Comm. Math. Phys. 185 (1997) 1--36.

%\bibitem{HMNS} G.P. Galdi, An Introduction to the Mathematical Theory of the Navier-Stokes Equations: Steady-State Problems, Springer Science \& Business Media, 2011.

\bibitem{our} Y. Gao, J.-G. Liu and J. Lu, Continuum limit of a mesoscopic model with elasticity of step motion on vicinal surfaces, Journal of Nonlinear Science 27 (2017), no. 3, 873-926.


%\bibitem{our1} Y. Gao, J.-G. Liu and J. Lu, Continuum limit of step evolution on vicinal surface in the ADL regime, submitted.

\bibitem{G2} Y. Gao, J.-G. Liu, and X. Y. Lu, Gradient flow approach to an exponential thin film equation: global existence and latent singularity, ESIAM Control Optim. Calc. Var., to appear.

  \bibitem{our2} Y. Gao, J.-G. Liu and J. Lu, Weak solution of a continuum model for vicinal surface in the attachment-detachment-limited regime, SIAM Journal on Mathematical Analysis 49 (2017), no. 3, 1705-1731.

  \bibitem{ourxu} Y. Gao, J.-G. Liu , X. Y. Lu and X. Xu, Maximal monotone operator theory and
	% in non-reflexive Banach space and the
	its applications to thin film equation in epitaxial growth on vicinal surface, Calculus of Variations and Partial Differential Equations 57 (2018), no. 2, 55.
	
\bibitem{Gao-Ji}
Y. Gao, H. Ji, J.-G. Liu and T. P. Witelski, A vicinal surface model for epitaxial growth with logarithmic free energy, Discrete \& Continuous Dynamical Systems-B (2018), 1771-1784.





%\bibitem{Grin1986} M. Grinfeld, Instability of the separation boundary between a nonhydrostatically stressed elastic body and a melt, Soviet Physics Doklady, 1986, 831--834.

\bibitem{Giga-Giga2010}
M.H. Giga and Y.  Giga,
Very singular diffusion equations: second and fourth order problems
Japan J. Indust. Appl. Math., 27 (2010), pp. 323-345.

\bibitem{giga2010} Y. Giga and R. V. Kohn, Scale-invariant extinction time estimates for some singular diffusion equations, Discrete and Continuous Dynamical Systems - A, 30(2011), no. 2, 509-535.

    \bibitem{Giga-Kuroda} Y. Giga, H. Kuroda and H. Matsuoka, Fourth-order total variation flow with dirichlet condition: Characterization of evolution and extinction time estimates,  (2015).

    \bibitem{Giga-Monika} Y. Giga, M. Muszkieta and P. Rybka, A duality based approach to the minimizing total variation flow in the space $H^{-s}$, arXiv preprint arXiv:1706.03223 (2017).

\bibitem{M-n} R. Granero-Belinch\'on and M. Magliocca, Global existence and decay to equilibrium for some crystal surface models. Discrete \& Continuous Dynamical Systems-A, 39 (2019), no. 4, 2101-2131.

%\bibitem{GoSe1964}
%C. Goffman and J. Serrin, Sublinear functions of measures and variational
%integrals. Duke Math. J. 31 (1964), 159-178.

%\bibitem{Guo:88} M. Z. Guo, G. C. Papanicolaou, and S. R. S. Varadhan, Nonlinear diffusion limit for a system with nearest neighbor interactions, Communications in Mathematical Physics, 118 (1988) 31--59.

%\bibitem{Isr2000} N. Israeli and D. Kandel, Decay of one-dimensional surface modulations, Physical review B 62 (2000), no. 20, 13707.

\bibitem{SSR2} H.-C. Jeong and E. D. Williams, Steps on surfaces: Experiment and theory, Surface Science Reports 34 (1999), no. 6, 171-294.

\bibitem{Kohnbook} R. V. Kohn, ``Surface relaxation below the roughening temperature: Some recent progress and open questions," Nonlinear partial differential equations: The abel symposium 2010, H. Holden and H. K. Karlsen (Editors), Springer Berlin Heidelberg, Berlin, Heidelberg, 2012, pp. 207-221.

%\bibitem{D22}
%R. V. Kohn, E. Versieux, Numerical analysis of a steepest-descent PDE model for surface relaxation
%below the roughening temperature, SIAM J. Num. Anal. 48 (2010) 1781-1800.

\bibitem{cooper1996} B. Krishnamachari, J. McLean, B. Cooper, J. Sethna, Gibbs-Thomson formula for small island sizes: corrections for high vapor densities,
Phys. Rev. B 54 (1996), 8899-8907.

\bibitem{D24}
J. Krug, H. T. Dobbs, S. Majaniemi, Adatom mobility for the solid-on-solid model, Z. Phys. B 97
(1995) 281-291.

%\bibitem{Tersoff1998} F. Liu, J. Tersoff and M. Lagally,
%  Self-organization of steps in growth of strained films on vicinal
%  substrates, Physical Review Letters 80 (1998), no. 6, 1268.

\bibitem{LLDM}
J.-G. Liu, J. Lu, D. Margetis  and J. L. Marzuola, Asymmetry in crystal facet dynamics of homoepitaxy by a continuum model, Physica D: Nonlinear Phenomena, to appear.

\bibitem{Liu-Bob} J.-G. Liu and R.M. Strain,
Global stability for solutions to the exponential PDE describing epitaxial growth,
Interfaces and Free Boundaries, to appear.
%\bibitem{jinhuan} J.-G. Liu and J. Wang, Global existence for a thin film equation with subcritical mass, submitted.

\bibitem{LX}
J.-G. Liu and X. Xu,
{\em Existence theorems for a multi-dimensional crystal surface model},
SIAM J. Math. Anal., 48 (2016) 3667--3687

%\bibitem{LuLiuMargetis:15} J. Lu, J.-G. Liu and D. Margetis, Emergence of step flow from an atomistic scheme of epitaxial growth in 1+1 dimensions, Physical Review E 91 (2015), no. 3, 032403.

%\bibitem{Luo2016} T. Luo, Y. Xiang and N. K. Yip, Energy scaling and
%  asymptotic properties of step bunching in epitaxial growth with
 % elasticity effects, Multiscale Modeling \& Simulation 14 (2016),
 % no. 2, 737--771.

%\bibitem{Luoinprep} T. Luo, Y. Xiang and N. K. Yip, private
 % communication.


%\bibitem{Majda2002} A. J. Majda and A. L. Bertozzi, Vorticity and
%  incompressible flow, vol. 27, Cambridge University Press, 2002.

\bibitem{Margetis2006} D. Margetis and R. V. Kohn, Continuum
  relaxation of interacting steps on crystal surfaces in $2+1$
  dimensions, Multiscale Modeling \& Simulation 5 (2006), no. 3,
  729--758.

%\bibitem{Margetis2011} D. Margetis, K. Nakamura, From crystal steps to continuum laws: Behavior near large facets in one dimension, Physica D, 240 (2011), 1100--1110.

\bibitem{MarzuolaWeare2013} J. L. Marzuola and J. Weare, Relaxation of
  a family of broken-bond crystal-surface models, Phys. Rev. E, 88
  (2013), 032403.

%\bibitem{Nishikawa:02} T. Nishikawa, Hydrodynamic limit for the
%  Ginzburg-Landau $\nabla \phi$ interface model with a conservation
%  law, J. Math. Sci. Univ. Tokyo 9 (2002), 481--519.

%\bibitem{SSR1}
%R. Najafabadi and D. J. Srolovitz, Elastic step interactions on vicinal surfaces of fcc metals, Surface Science 317 (1994), no. 1, 221-234.

\bibitem{widom1982}
J. S. Rowlinson, B. Widom, Molecular Theory of Capillarity, Clarendon Press, Oxford, 1982.

%\bibitem{Phelps2009} R. R. Phelps, Convex functions, monotone
 % operators and differentiability, vol. 1364, Springer, 2009.
 \bibitem{Zang1990} M. Ozdemir and A. Zangwill, Morphological equilibration
 of a corrugated crystalline surface, Physical Review B 42 (1990), no. 8, 5013-5024.

\bibitem{PimpinelliVillain:98} A. Pimpinelli and J. Villain, Physics
  of Crystal Growth, Cambridge University Press, New York, 1998.

%\bibitem{Rock1969} R. T. Rockafellar, Convex functions, monotone
 % operators and variational inequalities, Theory and Applications of
  %Monotone Operators, A. Ghizzetti (1969), 35--65.

\bibitem{OTAM} F. Santambrogio,
Optimal Transport for Applied Mathematicians, Springer, New York, 2015.


%\bibitem{Shenoy2002} V. Shenoy and L. Freund, A continuum description
 % of the energetics and evolution of stepped surfaces in strained
 % nanostructures, Journal of the Mechanics and Physics of Solids 50
 % (2002), no. 9, 1817--1841.

%\bibitem{Sidi1988} A. Sidi and M. Israeli, Quadrature methods for
 % periodic singular and weakly singular fredholm integral equations,
 % Journal of Scientific Computing 3 (1988), no. 2, 201--231.

%\bibitem{Sro1989} D. J. Srolovitz, On the stability of surfaces of
 % stressed solids, Acta Metallurgica 37 (1989), no. 2, 621--625.

\bibitem{Tang1997} L.-H. Tang, Flattening of grooves: From step
  dynamics to continuum theory, Dynamics of crystal surfaces and
  interfaces, Springer, New York, 1997.

%\bibitem{Tersoff1995} J. Tersoff, Y. Phang, Z. Zhang and M. Lagally,
%  Step-bunching instability of vicinal surfaces under stress, Physical
%  Review Letters 75 (1995), no. 14, 2730.

%\bibitem{villani} B. C. Villani, Topics in optimal transportation., volume 58 of graduate studies in mathematics,  (2010).

%\bibitem{WeeksGilmer:79} J. D. Weeks and G. H. Gilmer, Dynamics of
%  Crystal Growth, in Advances in Chemical Physics, Vol. 40, edited by
%  I. Prigogine and S. A. Rice (John Wiley, New York, 1979),
%  pp. 157--228.

\bibitem{Xiang2002} Y. Xiang, Derivation of a continuum model for
  epitaxial growth with elasticity on vicinal surface, SIAM Journal on
  Applied Mathematics 63 (2002), no. 1, 241--258.
  
  \bibitem{Xu-n} X. Xu, Existence theorems for a crystal surface model involving the p-Laplace operator. SIAM Journal on Mathematical Analysis 50 (2018), no. 4, 4261-4281.

%\bibitem{Xiang2004} Y. Xiang and W. E, Misfit elastic energy and a
%  continuum model for epitaxial growth with elasticity on vicinal
%  surfaces, Physical Review B 69 (2004), no. 3, 035409.


%\bibitem{Xiang2009} H. Xu and Y. Xiang, Derivation of a continuum
%  model for the long-range elastic interaction on stepped epitaxial
%  surfaces in $2+1$ dimensions, SIAM Journal on Applied Mathematics 69
%  (2009), no. 5, 1393--1414.


%\bibitem{Yau:91} H.-T. Yau, Relative entropy and hydrodynamics of
%  Ginzburg-Landau models, Lett. Math. Phys. 22 (1991), 63--80.



%\bibitem{Zangwill:88} A. Zangwill, Physics at Surfaces, Cambridge
%  University Press, New York, 1988.


\end{thebibliography}
\end{document}